\documentclass[12pt]{article}
\usepackage{graphicx}
\usepackage{latexsym}
\usepackage{amsmath}
\usepackage{verbatim}
\usepackage{amssymb}
\usepackage{mathrsfs}
\usepackage{amsmath,amssymb}
\DeclareMathAlphabet      {\mathbf}{OT1}{cmr}{bx}{n}
\DeclareFontFamily{OT1}{pzc}{}
\DeclareFontShape{OT1}{pzc}{m}{it}%
{<-> s * [1.15] pzcmi7t}{}
\DeclareMathAlphabet{\mathpzc}{OT1}{pzc}{m}{it}

\usepackage{ dsfont }
\usepackage{ mathrsfs }
\usepackage{enumitem}
\usepackage{amsbsy }
\usepackage{graphics}

\usepackage{amsthm}
\newtheorem{thm}{Theorem}
\newtheorem{theorem}{Theorem}[section]

\newtheorem{lemma}[theorem]{Lemma}
\newtheorem{corollary}[theorem]{Corollary}

\newtheorem{definition}[theorem]{Definition}

\newtheorem{remark}{Remark}[section]

\usepackage{amscd}
\usepackage{romannum}
\usepackage[toc,page,title,titletoc,header]{appendix}
\usepackage{hyperref}
\usepackage{tikz-cd}
\usepackage[all,cmtip]{xy}
\usetikzlibrary{matrix,calc}

\usepackage{lipsum}

\usepackage{abstract}

\begin{document}    \small
\pagenumbering{arabic}
\title{On   aspherical symplectic fillings  with finite  capacities of the prequantization bundles}
\author{{ Guanheng Chen}}
\date{}
\maketitle
\thispagestyle{empty}
\begin{abstract}    
 A  prequantization bundle is  a  negative circle bundle over a symplectic surface  together with a   contact form induced by a $S^1$-invariant connection.   Given  a symplectically  aspherical  symplectic filling  of a   prequantization bundle  satisfying certain topological  conditions,   suppose that a version of   symplectic capacity of the symplectic  filling is finite. Then, we show that the symplectic filling is diffeomorphic to the associated disk bundle. 
\end{abstract}

\section{Introduction and main result}
Let $(Y, \xi)$ be a contact manifold, where $\xi$ is a contact structure.   A \textbf{strong symplectic filling} of $(Y, \xi)$   is a symplectic manifold $(X, \omega)$ such that  $Y =\partial X$ and $\omega(V, \cdot)$ on $Y$ is a contact form of $\xi$, where $V$ is the Liouville vector field.  
 One of the  fundamental problems  in symplectic geometry is  try to  classify the symplectic fillings of a class of   contact manifolds or at least to find constraints of them. 

Prequantization bundles   are a class of contact manifolds which   have simple and well studied Reeb dynamics.     They are also called the Boothby–Wang bundles in many literatures.  Roughly speaking, a  prequantization bundle is  a circle bundle over a symplectic manifold  together with a contact form induced by a connection.  
When the base surface is the sphere, the prequantization bundle topologically is the lens space $L(p, 1)$.  In this case, the symplectic fillings of $L(p, 1)$ are classified by D. Mcduff \cite{M1}.  When the base manifolds are rational ruled surfaces,   the Stein fillability is studied by  M. Kwon and T. Oba \cite{KO}.   Similar to the case of the prequantization bundles, for a circle bundle over a  surface with $S^1$-invariant contact structure,  if the dividing set is nonempty, then the symplectic fillings are classified by S. Lisi, J. Van Horn-Morris, C.Wendl \cite{LVW1, LVW2}.




In  this note, we    study the case that  symplectic fillings of a prequantization bundle over  a surface with positive genus. We determine the diffeomorphism type of the symplectic fillings under some simple  topological conditions and a finiteness  assumption on a version of symplectic  capacity. 
 In contrast to the results in \cite{LVW1, LVW2}, the dividing set of  a prequantization bundle is empty. 
 
To give the statement of our main result, let us quick review the definition of   prequantization bundles. 
Let $(\Sigma, \omega_{\Sigma})$ be a closed symplectic manifold. Assume that $[\omega_{\Sigma}]\in H^2(\Sigma, \mathbb{R}) \cap H^2(\Sigma, \mathbb{Z})$ is integral.  Then we have a  complex   line bundle $\pi_E: E \to \Sigma$ such that  $c_1(E) =-[\omega_{\Sigma}]$.  Fix a Hermitian metric. Let $\pi: Y:=\{ r =1\} \to \Sigma$ be the  unit circle subbundle of $E$, where $r$ is the radius coordinate.   Let   $A$  be  a Hermitian connection 1-form    such that $\frac{i}{2\pi} F_{A} = -\omega_{\Sigma}$, where $ F_{A} $ is the curvature of $A$.  Then we have   a global angular form $\lambda \in \Omega^1(E -\Sigma, \mathbb{R})$, which is defined by  $\lambda : =\frac{1}{2\pi} (d\theta  - i A  \vert_U)$ under a unitary  trivialization $U \times \mathbb{C}$,  where  $d \theta$ is the angular form of $\mathbb{C}$ and $ A  \vert_U$ is a $i\mathbb{R}$ valued $1$-form.  $\lambda$ is well defined globally. We abuse the notation $\lambda$ to denote its restriction on $Y$.    The pair $(Y, \lambda)$ is called a  \textbf{prequantization bundle}.   Let    $\mathbb{D}E:=\{ r \le 1\}$  denote  the associated disk bundle. For our purpose, we assume $\Sigma$ is a surface with positive genus throughout. 

The main result of the note is as follows: 
\begin{thm} \label{thm0}
Let $(Y, \lambda)$ be the prequantization bundle over a  surface $\Sigma$ with genus $g(\Sigma) \ge 1$. Let  $\xi:=\ker \lambda$ denote the contact structure and   $e \le -1$ denote the Euler number of $\pi: Y \to \Sigma$. Let  $(X, \omega)$ be a  strong symplectic filling of $(Y, \xi)$. Assume that  $X$ satisfies the following topological constraints: 
\begin{enumerate} [label=\textbf{T.\arabic*}]
 \item \label{T1}
The fibers    of $Y$   are   contractible   in $X$. 
\item \label{T2}
$X$ is symplectically aspherical, i.e., $\omega \vert_{\pi_2(X)}=c_1(TX) \vert_{\pi_2(X)}=0$. 
\item   \label{T4}
The intersection form $Q_X: H_2(X, \mathbb{Z}) \times H_2(X, \mathbb{Z}) \to \mathbb{Z}$ is  negative definite, ie., $A \cdot A \le 0$ for any  class $A \in H_2(X, \mathbb{Z})$.  
\end{enumerate}
 Suppose that the symplectic capacity  $c_1(X, \omega) < \infty$ (see Definition \ref{def1}).  
 Then $X$ is diffeomorphic to the associated disk bundle $\mathbb{D}E$ over $\Sigma$. 
\end{thm}


The symplectic capacity $c_1(X, \omega)$ in Theorem \ref{thm0}  in fact is one of the higher symplectic capacities  $\mathfrak{g}_1^{\le 1}(X, \omega)$  defined by   K. Siegel \cite{KS}.   It should be equivalent  to the first Gutt-Hutchings capacity in \cite{GH}.  We   explain it a little more   in Remark \ref{remark2}.

\ref{T4} seems to be a strong assumption, but by a result of T.-J. Li, C. Y. Mak,   K. Yasui \cite{LMY},  if the Euler number  $e$ is small enough, then the \ref{T4} is true.  Hence, we have the following corollary. 
\begin{corollary}
 Suppose that $|e|+ \chi(\Sigma)>0$. Let  $(X, \omega)$ be a symplectic filling of $(Y, \xi)$ satisfying \ref{T1} and \ref{T2}.  If $c_1(X, \omega) < \infty$,   
 then $X$ is diffeomorphic to the associated disk bundle $\mathbb{D}E$ over $\Sigma$. 
 \end{corollary}
 \begin{proof}
Let $\pi^*:  E^* \to \Sigma $  be the dual  line bundle  of $E.$ Thus, $c_1(E^*)=-c_1(E)=[\omega_{\Sigma}]$ and $<c_1(E^*), \Sigma>=|e|$. We can endow $E^*$ with a symplectic  form  $\Omega_{E^*}$ such that the unit disk subbundle $(\mathbb{D} E^*, \Omega_{E^*})$ is a symplectic cap  of $(Y, \lambda)$. For the details of the construction, we refer the reader to  \cite{KO}.  The zero section in $E^*$, still denoted by $\Sigma$, is an embedded surface such that $[\Sigma] \cdot [\Sigma] =<c_1(E^*), \Sigma>=|e|>-\chi(\Sigma)$.  Therefore, $(\mathbb{D} E^*, \Omega_{E^*})$  is an adjunction cap in the sense of \cite{LMY}.  By Proposition 2.26 of \cite{LMY}, $b_2^+(X) =0$, i.e., $Q_X$ is negative definite.   Then the statement is a consequence of Theorem \ref{thm0}.
 \end{proof}
 
 The  idea  of the proof is to construct certain holomorphic foliation  on  $(X, \omega)$ and identify $X$ with certain moduli space of holomorphic curves. Then  we use this to show that $X$ has a bundle structure.   The methods here comes  from   D. Mcduff when she classifies the   symplectic ruled 4-manifolds \cite{M1}, and later developed by C. Wendl  to study symplectic fillings \cite{Wen3, LVW1, LVW2}.  Comparing with  \cite{Wen3, LVW1, LVW2},  the case here is simpler because the fibration structure does not have singularities. 
 
  \begin{remark}
 Our results only for four dimension, because the proof of Theorem \ref{thm0}  relies on the intersection theory of holomorphic curves which only works for four dimension. Lemmas \ref{lem1},\ref{lem2},\ref{lem7} should be generalized  to higher dimension by suitable modifying their proof. 
\end{remark}
 
 \begin{remark}
In many literatures, symplectically aspherical means $\omega\vert_{\pi_2(X)} =0$. Our definition in \ref{T2} is stronger than the usual one. We  need $c_1(TX)\vert_{\pi_2(X)} =0$ to define a $\mathbb{Z}$-grading on the linearization contact homology. 
\end{remark}

\section{Preliminaries}
 Let $\lambda$ be a contact form on $Y^{2n-1}$ such that $\lambda \wedge (d\lambda)^{n-1}>0$.  The \textbf{Reeb vector field} $R$ of $(Y, \lambda)$ is characterized by  conditions $\lambda(R)=1$ and $d\lambda(R, \cdot)=0$.
A \textbf{Reeb orbit} is a smooth map $\gamma: \mathbb{R}_{\tau} / T \mathbb{Z} \to Y $ satisfying the ODE $\partial_{\tau} \gamma =R \circ \gamma$ for some $T>0$.  The \textbf{contact action} of $\gamma$  is defined by
\begin{equation*}
\mathcal{A}_{\lambda}(\gamma) :=T= \int_{\gamma} \lambda.
\end{equation*}
A Reeb orbit is called \textbf{nondegenerate} if $1$ is not an eigenvalue of the linearized Reeb flow.  
Given $L \in \mathbb{R}$, the contact form $\lambda$ is called $L$-\textbf{nondegenerate} if all Reeb orbits with action  less than $L$ are nondegenerate. In three dimension,   each  nondegenerate   Reeb orbit  is classified as one of the following three types: elliptic,  positive hyperbolic  and negative hyperbolic.

\subsection{Rational holomorphic curves} 
Let $({X},  {\omega})$ be a symplectic cobordism with convex and concave boundaries.  
Let $(\widehat{X}, \hat{\omega})$ denote  the symplectic completion of $(X, \omega)$ by adding cylindrical ends. An almost complex structure $J$ on $\mathbb{R} \times Y$ is called \textbf{admissible} if $J$ is $\mathbb{R}$-invariant, $J(\partial_s) =R$ and $J \vert_{\xi}$ is $d\lambda$-compatible.   An almost complex structure $J$ on $\widehat{X}$ is called  \textbf{cobordism admissible} if $J$ is admissible on the cylindrical ends  and $J \vert_{X}$ is $\omega$-compatible. 

Fix an admissible complex structure $J$. A \textbf{rational holomorphic curve}  is a map $u$ from a punctured  Riemann sphere  $(\mathbb{S}^2-\Gamma, j)$ to   $\mathbb{R} \times Y$  or $\widehat{X}$  satisfying the Cauchy-Riemann equation $\bar{\partial}_{J, j} u=0$.   Define a piecewise 2-form $\check{\omega}$ by 
\begin{equation*}
\check{\omega}: = d\lambda_{+} \vert_{\mathbb{R}_{\ge0} \times Y_+} + \omega \vert_X + d\lambda_{-} \vert_{\mathbb{R}_{\le 0} \times Y_-}. 
\end{equation*}
The \textbf{energy} of a holomorphic curve is $\int \mathbf{u}^*\check{\omega}.$ It is easy to check that the energy of a holomorphic curve is nonnegative.

Given Reeb orbits $\gamma_1^+, ..., \gamma_{s_+}^+$ and   $\gamma_1^-, ...,\gamma_{s_-}^-$, let $ \mathcal{M}_{Y, m}^J(\gamma^+_1,..,\gamma_{s_+}^+;  \gamma^-_1,..,\gamma_{s_-}^-)$ denote  the moduli space of rational holomorphic curves in $\mathbb{R} \times Y$ with  positive ends at $\gamma^+_1,...,\gamma_{s_+}^+$ and negative ends at  $\gamma^-_1,...,\gamma_{s_-}^-$, and $m$ is the number of the marked points.  The moduli  space of rational  holomorphic  curves in $\widehat{X}$ is defined similarily, denoted by $\mathcal{M}_{X, m}^J(\gamma^+_1,..,\gamma^+_{s_+}; \gamma^-_1,..,\gamma^-_{s_-})$.  If $m=0$, we omit the subscript  $m$.  
Let $\overline{\mathcal{M}_{Y, m}^J}(\gamma^+_1,..,\gamma^+_{s_+}; \gamma_1^{-1},..., \gamma^-_{s_-})$  and $\overline{\mathcal{M}_{X, m}^J}(\gamma^+_1,..,\gamma^+_{s_+}; \gamma_1^{-},..., \gamma^-_{s_-})$  denote the compactification of the moduli space of holomorphic curves in the sense of  SFT \cite{BEHWZ}.  The virtual dimension of the moduli space is 
\begin{equation*}
\begin{split}
  dim \overline{\mathcal{M}_{X, m}^J}(\gamma^+_1,..,\gamma^+_{s_+}; \gamma^-_1,..,\gamma^-_{s_-}) = (n-3)\chi(u)+  2c_{\tau}(u^*TX) +  \sum_{i=1}^{s_+} \mu_{\tau}(\gamma_i^+) - \sum_{j=1}^{s_-} \mu_{\tau}(\gamma_j^- ) +2m,\\
\end{split}
\end{equation*}
where   $\chi(u)$ is the Euler number of the domain, $c_{\tau}(u^*TX)$ is the relative Chern number and   $\mu_{\tau}$ is the Conley-Zehnder index (see \cite{H2} for their definition).   

\begin{remark}
To take the SFT compactification, usually  one needs to fix a relative homology class $Z \in H_2(X,\gamma^+_1,...,\gamma^+_{s_+}; \gamma^-_1,...,\gamma^-_{s_-})$; otherwise, the energy of the holomorphic curves may not be bounded. Here  $H_2(X,\gamma^+_1,...,\gamma^+_{s_+}; \gamma^-_1,...,\gamma^-_{s_-})$ is the  set of 2-chains  in $X$ such that $\partial Z= \sum_{i=1}^{s_+} \gamma_i^+- \sum_{j=1}^{s_-} \gamma_i^-$, modulo boundary of 3-chains.   

Since we assume that $X$ is symplectically aspherical and the Reeb orbits are contractible in $X$, the energy of the rational  holomorphic curves only depend on $\gamma^+_1,...,\gamma^-_{s_-}$ (see the proof of Lemma \ref{lem10}).   So the SFT compactification makes sense for the whole moduli space. 
\end{remark}

Let $ \overline{\mathcal{M}^J}$ denote the  compactification of the moduli space of rational curves (possibly with point constraints).  In general,  $ \overline{\mathcal{M}^J}$  is  not a manifold or an orbifold. To define the linearized contact homology latter, we need  a notion of count of these moduli spaces.    This goal can  be achieved by    several  approaches, either polyfold approaches by  H. Hofer, K. Wysocki, E. Zehnder \cite{HWZ},   or the virtual fundamental cycles by J. Pardon \cite{P1, P2},  or Kuranishi approaches by S. Ishikawa \cite{I}. 
 
  Assume that  there exists  a virtual perturbation scheme  (using one of the methods mentioned above) so that we can assign a virtual count  $\#^{vir} \overline{\mathcal{M}^J} \in \mathbb{Q}$ to the  moduli space.   Moreover,  the virtual count satisfies the following axioms:
\begin{enumerate}
\item
$\#^{vir} \overline{\mathcal{M}^J} =0$ unless $dim  \overline{\mathcal{M}^J}=0$. 
\item 
If $\mathcal{M}^J $ is  compact and transversely cut out, then  $\#^{vir} \overline{\mathcal{M}^J} = \# \mathcal{M}^J.$ Therefore, $ \#^{vir} \overline{\mathcal{M}^J} =0$ if it is empty. 
\item
If $dim \overline{\mathcal{M}^J}  =1$, then $\#^{vir} \partial \overline{\mathcal{M}^J} =0$ . 
\end{enumerate}

\subsection{Linearized contact homology}  The linearized contact homology is introduced by  F. Bourgeois and A. Oancea which is an alternative model of the  positive $S^1$-equivariant symplectic Floer homology \cite{BO1}.  
 
 Let $(Y, \lambda)$   be a contact manifold  such that $\lambda$   is nondegenerate. Also, we assume that $(X, \omega)$ is a symplectically aspherical filling such that $\lambda =\omega(V, \cdot)$ over $Y$, where $V$ is the Liouville vector field. 
  Let $\mathcal{P}_c(Y, \lambda)$ denote the Reeb orbits which are contractible   in $X$.     
A  Reeb orbit $\gamma$ is \textbf{bad} if it is an $k$-fold cover of a simple orbit $\gamma'$ such that $k$ is even and $\mu_{\tau}(\gamma) -\mu_{\tau}(\gamma')$ is odd.   Reeb orbits that are not bad are called \textbf{good}.

Let $C_*(X, \omega)$ be the    $\mathbb{Q}$-module  generated freely  by the good  orbits in  $\mathcal{P}_c(Y, \lambda)$.  
Define an augmentation $\epsilon_X:  C_*(X, \omega) \to \mathbb{Q}$ by 
\begin{equation}
\epsilon_X(\gamma): =\#^{vir} \overline{\mathcal{M}_X^J}(\gamma; \emptyset),
\end{equation}
where $\overline{\mathcal{M}_X^J}(\gamma; \emptyset)$ is the moduli space of holomorphic planes   with a positive end at $\gamma$. The differential on $C_*(X, \omega)$ is defined by 
\begin{equation}
\partial \gamma_+: = \sum_{\gamma_-} \sum_{\gamma_1,...,\gamma_k}\#^{vir} \left( \overline{\mathcal{M}_Y^J}(\gamma_+;  \gamma_-, \gamma_1,..., \gamma_k) /\mathbb{R} \right) \epsilon_X(\gamma_1)...\epsilon_X(\gamma_k) \gamma_-. 
\end{equation}
The linearized contact homology $CH(X, \omega)$ is the homology of $(C_*(X, \omega), \partial)$.

\paragraph{Grading and action filtration}
By the symplectically aspherical assumption \ref{T2}, we  can define a  $\mathbb{Z}$-grading and an   action  filtration on $CH(X, \omega)$. 

 Let $\gamma \in  \mathcal{P}_c(Y, \lambda)$ be a Reeb orbit. Since $\gamma$ is contractible  in $X$, we can find a disk $\sigma: \mathbb{D} \to X$ such that $\sigma(\partial \mathbb{D}) =\gamma$.    For any trivialization $ \tau$ of $\gamma^*\xi$, we  extend it to $\sigma^*TX $.  Define  
$$CZ(\gamma) : = 2c_{\tau}(\sigma^*TX)+ \mu_{\tau}(\gamma). $$
Note that $CZ(\gamma)$ is independent of the choice of $\tau$.  The \textbf{SFT grading} of a 
contractible     Reeb orbit is  $|\gamma|:=(n-3)+CZ(\gamma).$
We also define the  \textbf{symplectic action} of $\gamma$ by 
$$A_{\omega}(\gamma):=\int \sigma^*\omega.$$
Because $\omega \vert_{\pi_2(X)}=c_1(TX) \vert_{\pi_2(X)}=0$
,  $CZ(\gamma)$ and $A_{\omega}(\gamma)$ are  independent of the choice of  the disk $\sigma$.  
\begin{remark}
Unlike  the contact action $\mathcal{A}_{\lambda}$ of Reeb orbits, the symplectic action $A_{\omega}(\gamma)$ here is not necessarily nonnegative. Unless the symplectic filling is exact,     in general we may have $\mathcal{A}_{\lambda}(\gamma) \ne A_{\omega}(\gamma)$. 
\end{remark}
 
 Let $C_*^L(X, \omega)$ be the chain complex generated by the good   orbits  in $\mathcal{P}_c(Y, \lambda)$ with symplectic action  $A_{\omega}$ less than $L$. 
\begin{lemma} \label{lem10}
$C_*^L(X, \omega)$ is a subcomplex. 
\end{lemma}
\begin{proof}

Suppose that $<\partial \gamma_+, \gamma_-> \ne 0$. Then we have a   broken holomorphic curve $\mathbf{u} \in   \overline{\mathcal{M}_Y^J}(\gamma_+; \gamma_-, \gamma_1,...\gamma_k)$    in $\mathbb{R} \times Y$ and holomorphic planes  $\mathbf{u}_i \in \overline{\mathcal{M}_X^J}(\gamma_i, \emptyset)$ in $\widehat{X}$.  

Let $\sigma_{\pm}$ and $\{\sigma_i\}_{i=1}^k$ be the disks in $X$  bounded by $\gamma_{\pm}$ and $\{\gamma_i\}_{i=1}^k$  respectively.  Then $(-\sigma_+) \# \mathbf{u} \# \sigma_- \#_{i=1}^k\sigma_i$ defines  a spherical class in $\pi_2(X)$.  By the symplectically aspherical assumption, we have 
\begin{equation*}
\begin{split}
\int {\mathbf{u}}^*d\lambda=\int {\mathbf{u}}^*\check{\omega} &=\int \sigma_+^*\omega- \int \sigma_-^*\omega - \sum_{i=1}^k  \int \sigma_i^* \omega + \int \left((-\sigma_+) \# \mathbf{u} \# \sigma_- \#_{i=1}^k  \sigma_i \right)^*\omega\\
&=A_{\omega}(\gamma_+) -A_{\omega}(\gamma_-) -\sum_{i=1}^k A_{\omega}(\gamma_i). 
\end{split}
\end{equation*}
By definition and Stokes' theorem,  we have $
A_{\omega}(\gamma_i) = \int {\mathbf{u}_i^*}\check{\omega}. 
$
Because the energy of the holomorphic curves  are  nonnegative, we have
$$A_{\omega}(\gamma_i) \ge 0 \mbox{ and } A_{\omega}(\gamma_+) \ge A_{\omega}(\gamma_-) +\sum_{i=1}^k A_{\omega}(\gamma_i).$$  
In particular,  we have $A_{\omega}(\gamma_+) \ge A_{\omega}(\gamma_-)$. Therefore,  $C_*^L(X, \omega)$ is a subcomplex. 
\end{proof}

 \paragraph{Symplectic capacities}

Define $\epsilon_X\langle p \rangle: C_*(X, \omega) \to  \mathbb{Q}$  by 
\begin{equation}
\epsilon_X\langle p \rangle (\gamma) : =  \#^{vir} \overline{\mathcal{M}_X^J}(\gamma; \emptyset)\langle p \rangle,
\end{equation}
where $   \overline{\mathcal{M}_X^J}(\gamma; \emptyset)\langle p \rangle$ is the moduli space of holomorphic planes passing through $p$. 
 $\epsilon_X\langle p \rangle$ is a chain map, and hence it descends to $CH_*(X, \omega)$, still denoted by $\epsilon_X\langle p \rangle$. We  decompose the complex $C_*(X, \omega)$ by  the SFT grading. By dimensional reason, we have  $\epsilon_X\langle p \rangle: CH_{2n-2}(X, \omega)\to  \mathbb{Q}$.

\begin{definition} [see \cite{KS}] \label{def1}
The symplectic capacity  is defined by $$c_1(X, \omega):= \inf\{L:    \exists \ \sigma \in CH^L_*(X, \omega)  \mbox{ s.t. } \epsilon_X\langle p \rangle (\sigma) \ne 0 \}. $$
\end{definition}
If $(Y, \lambda)$ is degenerate,  we  find a sequence of smooth functions $\{f_n\}_{n=1}^{\infty}$ such that   $f_n \ge 1$, $f_n \lambda $ are nondegenerate,   and $\{f_n\}_{n=1}^{\infty}$ converges to  $1$ in $C^0$ topology.  Let $(X_n, \omega_n)$ be the composition of $(X, \omega)$ with the following      exact cobordisms  
\begin{equation} \label{eq4}
(X_{f_n}:= \{ (s, y) \in \mathbb{R} \times Y: 1 \le s \le f_n(y) \}, \omega_{f_n}:= d(e^s\lambda)). 
\end{equation}
Then we define the symplectic capacity by $$c_1(X, \omega): =\lim_{n \to \infty} c_1(X_n, \omega_n).$$
According to Corollary \ref{lem9} in the next subsection,  we know that the  limit is independent of the choice of the sequence $\{f_n\}_{n=1}^{\infty}$. 
 
\begin{remark}\label{remark2}
The symplectic capacity $c_1(X, \omega)$ is one of the Higher symplectic capacities defined by   K. Siegel \cite{KS}.  In \cite{KS}, the notation is $\mathfrak{g}_1^{\le 1}(X, \omega)$.  When $(X, \omega)$ is a Liouville domain,  one could expect that   $\mathfrak{g}_k^{\le 1}(X, \omega)$ agrees with the Gutt-Hutchings capacities $c_k^{GH}(X, \omega)$ which are defined by the $S^1$-equivariant symplectic Floer homology \cite{GH}.   The reason is that  Bourgeois and Oancea define  an isomorphism between the linearized contact homology and the  positive  $S^1$-equivariant symplectic Floer homology \cite{BO2}. Moreover, the isomorphism preserves the action filtration.  See Theorem 7.64 of  M. Pereira's thesis \cite{MP} for the details. 

V. L. Ginzburg and  J. Shon  generalize  the Gutt-Hutchings capacities    to the symplectically aspherical symplectic fillings by using the same frame work of J. Gutt and M. Hutchings \cite{GS}.  Therefore,  one can expect that  $c_1(X, \omega)$  also  agrees with first Gutt-Hutchings capacities  in  the  symplectically aspherical setting. 
\end{remark} 
 
\paragraph{Cobordism maps} Let $(X, \omega)$ be an exact  symplectic  cobordism from $(Y_+, \lambda_+)$ to $(Y_-, \lambda_-)$. 
Let $(X_-, \omega_-)$ be a  symplectic filling of  $(Y_-, \lambda_-)$.  The  composition of   $(X, \omega)$ and  $(X_-, \omega_-)$ is denoted by $(X_+, \omega_+)$.  Assume that $X_{\pm}$ are symplectically  aspherical. 

Define  a  map 
$\phi(X, \omega): C_*(X_+, \omega_+) \to  C_*(X_-, \omega_-) $ by 
\begin{equation*}
\phi(X, \omega)  \gamma_+: = \sum_{\gamma_-} \sum_{\gamma_1,...,\gamma_k}\#^{vir} \overline{\mathcal{M}_{X}^J}(\gamma_+; \gamma_-, \gamma_1,...,\gamma_k)  \epsilon_{X_-}(\gamma_1)...\epsilon_{X_-}(\gamma_k) \gamma_-. 
\end{equation*}
Here $\phi(X, \omega)$ is a chain map and it induces a homomorphism 
\begin{equation*}
\Phi(X, \omega): CH_*(X_+, \omega_+) \to  CH_*(X_-, \omega_-). 
\end{equation*}
The map $\Phi(X, \omega)$ is called a \textbf{cobordism map}. Moreover, it  satisfies  
\begin{equation*}
 \epsilon_{X_-}\langle p \rangle \circ \Phi(X, \omega) =  \epsilon_{X_+}\langle p \rangle. 
\end{equation*} 

\begin{lemma}
The homomorphism $\phi(X, \omega)$ maps  $C_*^L(X_+, \omega_+) $ to $  C_*^L(X_-, \omega_-)$. Therefore, we have 
\begin{equation*}
\Phi(X, \omega): CH_*^L(X_+, \omega_+) \to  CH_*^L(X_-, \omega_-). 
\end{equation*}
\end{lemma}
\begin{proof}
The proof is the same as Lemma 2.1. 
\end{proof}

\begin{corollary} \label{lem9}
Let $i: (X_-, \omega_-) \hookrightarrow (X_+, \omega_+)$ be a symplectic embedding. Suppose that $(X_{\pm}, \omega_{\pm})$ are symplectically aspherical and $(X, \omega)= (X_+ \setminus i(X_-), \omega_+)$ is exact.  Then $c_1(X_-, \omega_-)\le c_1(X_+, \omega_+).$
\end{corollary}
\begin{proof}
It is easy to check that we have the following diagram:
$$ \begin{tikzcd}
[row sep=large, column sep = 10ex]
  CH_*^L(X_+, \omega_+)  \arrow[r, "i_+^L"] \arrow[d, "{\Phi(X,  \omega)}"]
    & CH_*(X_+, \omega_+)  \ar[d,  "{\Phi(X,  \omega)}"]  \arrow[r, "\epsilon_{X_+}\langle p \rangle " ] &\mathbb{Q}  \arrow[d ]\\
  CH^L(X_-,  \omega_-) \arrow[r,  "i_-^L" ] &CH(X_-,  \omega_-) \arrow[r,  "\epsilon_{X_-}\langle p \rangle "  ]
&   \mathbb{Q}
\end{tikzcd}$$
Here $i^L_{\pm}: CH_*^L(X_{\pm}, \omega_{\pm}) \to CH_*(X_{\pm}, \omega_{\pm})$	are  the homomorphisms  induced by the inclusion 	$C_*^L(X_{\pm}, \omega_{\pm}) \to C_*(X_{\pm}, \omega_{\pm})$. 	
From the definition and the above diagram, we get  $c_1(X_-, \omega_-)\le c_1(X_+, \omega_+).$
\end{proof}
 
 
\paragraph{Intersection theory for holomorphic curves} Since we consider  the four dimensional case, the intersection theory plays a key role in our argument. Let us have a quick review of it. 

Let $u, u'$ be two  simple holomorphic curves. Suppose that $u/u'$ is asymptotic to Reeb orbit $\gamma_z/\gamma_z'$ at puncture $z/z'$.   The Reeb orbits here are either nondegenerate or Morse-Bott. R. Siefring defines an intersection number $u\bullet u' $  which is a homotopic invariant \cite{S}.  Later,  C. Wendl generalizes it to Morse-Bott setting \cite{Wen2}.   The precise definition of $u\bullet u' $  is as follows: 
\begin{equation*}
u\bullet u' =Q_{\tau}(u, u') - \sum_{(z, z') \in \Gamma_{\pm} \times \Gamma'_{\pm}} \Omega^{\tau}_{\pm}(\gamma_z \mp \epsilon, \gamma_{z'} \mp \epsilon),
\end{equation*}
 where  $Q_{\tau}(u, u')$ is the relative intersection number (see \cite{H2} for its definition), and  the definition of   $\Omega^{\tau}_{\pm}(\gamma_z \mp \epsilon, \gamma_{z'} \mp \epsilon)$ is explained as follows.   If $\gamma_z$ and $\gamma_z'$ are geometrically distinct, then $\Omega^{\tau}_{\pm}(\gamma_z \mp \epsilon, \gamma_{z'} \mp \epsilon):=0$. 
Let $\gamma$ be a simple Reeb orbit. For $\delta \in \mathbb{R}$,  define 
\begin{equation*}
 \Omega^{\tau}_{\pm}(\gamma^n + \delta, \gamma^m + \delta): =mn \min\{ \frac{\mp \alpha_{\mp}^{\tau} (\gamma^m+\delta)}{m},  \frac{\mp \alpha_{\mp}^{\tau} (\gamma^n+\delta)}{n} \}, 
\end{equation*}
where $ \alpha_{-}^{\tau} (\gamma^n+ \delta)= \lfloor \mu_{\tau}(\gamma^n + \delta) /2 \rfloor$   and $ \alpha_{+}^{\tau} (\gamma^n+ \delta)= \lceil \mu_{\tau}(\gamma^n + \delta) /2 \rceil$.   Let $\mathbf{A}_{\gamma}: \Gamma(\gamma^*\xi) \to \Gamma(\gamma^*\xi)$  denote the  asymptotic operator of $\gamma$ (see \cite{Wen2} for its definition).  The notation  $ \mu_{\tau}(\gamma \pm \delta) $
 means  $ \mu_{\tau}(\gamma \pm \delta):=\mu_{\tau} (\mathbf{A}_{\gamma} \pm \delta).$

We will use the following facts (For details, we refer reader to \cite{S, Wen2}):
\begin{enumerate} [label=\textbf{F.\arabic*}]
\item \label{F1}
If $u$ and $u'$ are geometric  distinct, then $u\bullet u' \ge 0$. If  $u\bullet u'  =0 $, then they are disjoint. 
\item\label{F2}
We have the following adjunction formula
\begin{equation*}
u\bullet u=  2\delta_{total}(u)+ \frac{1}{2}( ind u -2 +2g(u) +  \# \Gamma_0(u))  + \sum_{z\in\Gamma} cov_{\infty}(\gamma_z) +  \sum_{z\in\Gamma}  cov_{MB}(\gamma_z),
\end{equation*}
where $\Gamma_0(u)$ denote the punctures which are  asymptotic to   Reeb orbits with even Conley-Zehnder index,  and $\delta_{total}(u)$ is the total singularity index. 

\item \label{F3}
In fact, $\delta_{total}(u) = \delta(u)+\delta_{\infty}(u)$. Here  $\delta(u) \ge 0$ is an algebraic count of self-intersections  and singularities, and  $\delta(u) = 0$ if and only if $u$ is embedded. 
The other term  $\delta_{\infty}(u) \ge 0$ is the ``hidden double points at infinity''. 
If $u$ is asymptotic to distinct simple nondegenerate  Reeb orbits, then $\delta_{\infty}(u)=0$.  
\item \label{F4}
Suppose that $\{\gamma_z\}_{z \in\Gamma}$ lie inside distinct  Mose-Bott manifolds.   If $\gamma_z$   and $\gamma^{\epsilon}_z$ are simple,   then  $ cov_{\infty}(\gamma_z)=    cov_{MB}(\gamma_z)=0$, where $\gamma^{\epsilon}_z$ denotes  any nearby generic orbit in the same Morse-Bott family as $\gamma_z$.  In particular, this is true when the Morse-Bott manifold has minimal contact action. 

\end{enumerate}

\section{Proof of Theorem \ref{thm0}}
Let $\lambda$ be the contact form induced by the connection $A$ given in the Introduction.  Let  $V$ be  the Liouville vector field.
We can assume that $\lambda=\omega(V, \cdot)$ on $Y=\partial X$ by the following reasons.  

Since any two contact forms defining the same contact structure are differed by a positive function, we have $\omega(V, \cdot)=f\lambda$.  By rescaling  $\omega$, we can assume that $f>1$.  We still have $c_1(X, \omega)<\infty$ after rescaling, because $c_1(X, r\omega)=rc_1(X, \omega)$.  Note that  $X \setminus X_f$ still satisfies the assumptions \ref{T1}, \ref{T2}, \ref{T4}, where $X_f$ is given by (\ref{eq4}).  Moreover, by Corollary \ref{lem9}, $c_1(X \setminus X_f, \omega) \le c_1(X, \omega) < \infty.$ Thus, if we can show that  $X \setminus X_f$ is diffeomorphic to $\mathbb{D}E$, then the conclusion still true for the original $(X, \omega)$.  Therefore, we assume that $ \lambda=\omega(V, \cdot)$ on $Y$ from now on. 

To begin with, we  deduce a    topological constrain  of the  symplectic filling $(X, \omega)$  in Theorem \ref{thm0}. 
 \begin{lemma} 
Let $(X, \omega)$ be a  symplectic filling $(X, \omega)$  satisfying \ref{T1}. Then the inclusion $i_*:  H_2(Y, \mathbb{Z}) \to H_2(X, \mathbb{Z})$ is the zero map. 
 \end{lemma}
 \begin{proof}
By Lemma 3.7 of  \cite{NW},  each class in  $H_2(Y, \mathbb{Z}) $ is represented by $\pi^{-1}(\eta)$, where $\eta: S^1 \to \Sigma$ is a representative of a class in  $H_1(\Sigma, \mathbb{Z})$. Since each fiber in $\pi^{-1}(\eta)$  bounds a disk $\sigma_t$ in $X$, the union of the disks  $\cup_{t\in S^1} \sigma_t $ forms a solid torus whose boundary is $\pi^{-1}(\eta)$.  Then  $[\pi^{-1}(\eta)]=0 \in H_2(X, \mathbb{Z})$. 
\end{proof}


\subsection{Nondegenerate situation}
There should be a Morse-Bott version of  linearized contact homology, but  to the author's knowledge, such a version has not yet appeared in any literature. Now the contact form  of the prequantization  bundle is Morse-Bott. Thus,  to make use of the linearized contact homology, we first need to  make some perturbation.   
   \paragraph{Morse-Bott perturbations}
Under  a unitary trivialization, the Reeb vector field is $R=2\pi \partial_{\theta}$. Therefore, 
every fiber $\gamma_z =\pi^{-1}(z)$ is a Reeb orbit of $(Y, \lambda)$.  Also, their  iterations are the only Reeb orbits.  Thus,  the contact form $\lambda$ is Morse--Bott,  and the base $\Sigma$ is a  Morse-Bott submainfold.   
Fix a  perfect Morse function $H: \Sigma \to \mathbb{R}$.  By the standard Morse-Bott perturbation \cite{BO} (also see   Nelson and Weiler's  papers   \cite{NW} in our setting),    define a contact form 
\begin{equation*}
\lambda_{\varepsilon}: = (1+{\varepsilon}\pi^*H) \lambda,
\end{equation*}
where   $0 < \varepsilon \ll1  $ is a small fixed number. 

Let $p_{\pm}$ be the minimum, maximum of $H$,  and $\{p_i\}_{i=1}^{2g}$ be the saddle points of $H$ respectively.   Let $e_{\pm}:= \pi^{-1}(p_{\pm})$ and $h_i:= \pi^{-1}(p_i)$.   For any $L$, there exists $\varepsilon_L>0$ such that for $0<\varepsilon< \varepsilon_L$,  
\begin{itemize}
\item
$\lambda_{\varepsilon}$ is $L$-nondegenerate,
\item
$e_{\pm}$ and $h_i$ are the only simple  Reeb orbits of $\lambda_{\varepsilon}$ with $\mathcal{A}_{\lambda_{\varepsilon}}<L$.  
\end{itemize}  Moreover,  $e_{\pm}$ are elliptic orbits and $\{h_i\}_{i=1}^{2g}$ are positive hyperbolic orbits. All of these orbits and  their iterations  are good   orbits.  By definition, the contact action  of the Reeb orbits are 
\begin{equation*}
\mathcal{A}_{\lambda_{\varepsilon}} (\gamma_{p}^k) = \int_{\gamma_{p}^k} \lambda_{\varepsilon}= k(1+ \varepsilon H(p)).  
\end{equation*}

%

Let $(X, \omega)$ be a symplectically  aspherical symplectic filling of $(Y, \lambda)$.   Fix  $0< \varepsilon \ll1$. We have a symplectic form $\omega_{\varepsilon}$ such that $\omega_{\varepsilon} =\omega$ outside a collar neighbourhood of $Y$   and $\omega_{\varepsilon} =d(e^s \lambda_{\varepsilon})$ near $Y$.   Moreover,   $\omega_{\varepsilon}$ converges to $\omega$ in $C^{\infty}$-topology as $\varepsilon \to 0$.  The construction of    $\omega_{\varepsilon}$  is as follows.  Using the Liouville vector field, we identify a collar neighbourhood of $Y$ by $((-\delta_0, 0]_s \times Y, \omega = d (e^s \lambda))$. 
Let $\varepsilon(s)$ be a nondecreasing cut off function such that $\varepsilon(s)=\varepsilon \ll 1$ when  $s \ge  -\frac{3}{4} \delta_0$ and  $\varepsilon(s)=0$ when $s \le - \delta_0$. Define 
\begin{equation*} 
\omega_{\varepsilon} :=
\begin{cases}
\omega  & s  \le  -\delta_0 \\
d(e^s(1+ \varepsilon(s)\pi^*H) \lambda)  &    -\delta_0 \le s \le 0. 
\end{cases}
\end{equation*}
By Stokes' theorem,  $(X, \omega_{\varepsilon})$  is still  symplectically aspherical.  One can define a diffeomorphism $i_{\varepsilon}: \widehat{X} \to \widehat{X}$ by  $i_{\varepsilon}:=id$ when $s  \le  -\delta_0$, and 
\begin{equation*} 
i_{\varepsilon}(s, y):=(s+ \log(1+ \varepsilon(s)\pi^*H),  y)  \mbox{ when }  s \ge   -\delta_0. 
\end{equation*}
Then $i_{\varepsilon}^* \omega =\omega_{\varepsilon}$ on $X$. 
 By definition and Corollary \ref{lem9},  $c_1(X, \omega_{\varepsilon})$  is still  finite provided that $c_1(X, \omega)$ is finite. 
Fix a generic cobordism admissible   almost complex structure $J_{\varepsilon}$ of $(X, \omega_{\varepsilon})$.

\paragraph{Degree and index of holomorphic curves} Before we move on, let us  recall  the degree of holomorphic curves defined by J. Nelson and M. Weiler \cite{NW}. It is very useful in many argument later. 
Let $u: \mathbb{S}^2 -\Gamma \to \mathbb{R} \times Y$ be a rational curve. Since $\pi \circ u$ maps the punctures to points on $\Sigma$, one  can extend $\pi\circ u$  to a map $\pi\circ \bar{u}: \mathbb{S}^2 \to \Sigma.$ The \textbf{degree} of $u $ is $deg(u):= deg(\pi\circ \bar{u})$.  Because $\pi_2(\Sigma)=0$, the degree $deg(u)$ always vanishes in our situation.  Suppose that the positive ends of  $u$ are asymptotic to the fibers with total multiplicity $M$, and the  negative  ends of  $u$ are asymptotic to the fibers with total multiplicity $N$.  Nelson and   Weiler \cite{NW} show that $$deg(u)= (M-N)/|e|. $$  Therefore, we have $M=N$ for rational curves.  


\paragraph{Trivializtion}  As explained in \cite{NW}, any trivialization of $T_p\Sigma$  can be lifted to a trivialization $\tau$ of $\gamma_p^*\xi$.  Moreover, the linearized Reeb flow is the identity map with respect to $\tau$. We extend $\tau$ to be a  trivialization of $\sigma^*TX$ through the disk $\sigma$ bounded by $\gamma_p$, still denoted by $\tau$.  We fix  such a  trivialization $\tau$  throughout. 

\paragraph{Moduli space of rational curves} Because $c_1(X, \omega)$ is finite, there exists $L>0$ such that  $c_1(X, \omega_{\varepsilon})<L$ for any small $\varepsilon$.  We fix $\varepsilon>0$ such that $\varepsilon< \varepsilon_L$.

  Recall that the Reeb orbits $e^k_{\pm}$ and $h^k_i$ are contractible in $X$ by the assumption \ref{T1}. Thus they  have well-defined  Conley-Zehnder index and symplectic action.  The next three  lemmas deduce the precise values of the Conley-Zehnder index.     
\begin{lemma} \label{lem1}
For any $1 \le k < \lfloor L \rfloor$, the Conley-Zehnder indexes  satisfy   the following properties:
\begin{equation*}
\begin{split}
&CZ(e^k_+)=kCZ(e_-)+k+1, CZ(h^k_i)= kCZ(e_-)+k, \mbox{ and } CZ(e^k_-)=kCZ(e_-)+k-1.
\end{split}
\end{equation*}

\end{lemma}

\begin{proof}
Let $\sigma: \mathbb{D} \to X$  be the disk bounded by $e_-$.  Let $p\in \Sigma$ and  $\sigma_{p_-, p}: [0,1] \times S^1 \to Y$ be the cylinder corresponding to $\pi^{-1}({\eta})$, where $\eta$  is  a path line from  $p_-$ to $p$. Gluing $\sigma$ and $\sigma'$ along $e_-$ produces a disk $\sigma\# \sigma_{p_-, p}$ bounded by $\gamma_p=\pi^{-1}(p)$.

By the proof of Lemma 3.12 in \cite{NW}, we get $c_{\tau}((\sigma_{p_-, p})^*\xi) =0$. By Lemma 3.9 of  \cite{NW}, $\mu_{\tau}(e_{\pm}^k)=\pm 1$ and $\mu_{\tau}(h_i^k)=0$. Therefore, we have 
\begin{equation*}
\begin{split}
&CZ(e_-)=2c_{\tau}(\sigma^*TX) -1,\\
&CZ(e^k_-)=2kc_{\tau}(\sigma^*TX) -1=kCZ(e_-)+k-1,\\
&CZ(e^k_+)=2kc_{\tau}((\sigma\#\sigma_{p_-, p_+})^*TX) +1 =2kc_{\tau}(\sigma^*TX) +1=kCZ(e_-)+k+1,  \\
&CZ(h^k_i)=2c_{\tau}((\sigma\#\sigma_{p_-, p_i})^*TX)  =2c_{\tau}(\sigma^*TX) =kCZ(e_-)+k.\\
\end{split}
\end{equation*}
\end{proof}

\begin{lemma} \label{lem2}
Suppose that $c_1(X, \omega_{\varepsilon})< L $. Then either  $CZ(e_-)=1$ or $CZ(e_-)=3$. 
\end{lemma}
\begin{proof}
Assume that $CZ(e_-)=2c_{\tau}(\sigma^*TX)-1 \le -1$.  
By Lemma \ref{lem1}, we have 
\begin{equation*}
\begin{split}
&CZ(e^k_+)\le 1, CZ(h^k_i) \le 0,  \mbox{ and } CZ(e^k_-) \le -1.
\end{split}
\end{equation*}
As a result, $CH_2^{L}(X, \omega_{\varepsilon}) =0$.  This contradicts  the finiteness of $c_1(X, \omega_{\varepsilon})$. 

Since $CZ(e_-)$ is odd, we have $CZ(e_-) \ge 1$.  Suppose that $CZ(e_-) \ge  5$.   Lemma \ref{lem1} implies that 
\begin{equation*}
\begin{split}
&CZ(e^k_+)\ge 6k+1\ge 7, CZ(h^k_i)\ge 6k \ge 6,   \mbox{ and }  CZ(e^k_-)\ge 6k-1 \ge 5.
\end{split}
\end{equation*}
Again, $CH_2^L(X, \omega_{\varepsilon}) = 0$, and we have a contradiction. 
Since $CZ(e_-)$ is odd, we have either $CZ(e_-)=3$ or $CZ(e_-)=1$. 
\end{proof}

\begin{lemma} \label{lem7}
Suppose that $c_1(X, \omega_{\varepsilon})< L$. Then we must  have $CZ(e_-)=1$. 
\end{lemma}
\begin{proof}
It suffices to rule out the case that  $CZ(e_-)=3$. We prove this by contradiction argument. 
Assume $CZ(e_-)=3$. By Lemma \ref{lem1}, we know  that $|\gamma|=2$ if and only if $\gamma=e_-$.  Since  $c_1(X, \omega_{\varepsilon})<L$, we have  $$\epsilon_X\langle p \rangle(e_-)=\#^{vir} \overline{\mathcal{M}_X^{J_{\varepsilon}}}(e_-; \emptyset)\langle p \rangle \ne 0.$$ 
Fix a point $q \in Y$ away from the fibers at critical points of $H$. Let $\eta: [0, \infty)_{\tau} \to \widehat{X}$ be a path such that  $\eta(0)=p$ and  $\eta(\tau) =q \in   Y$ for  sufficient large  $\tau$. 
Let $ \overline{\mathcal{M}_X^{J_{\varepsilon}}}(e_-; \emptyset)\langle \eta \rangle$ be the moduli space of holomorphic planes   passing through $\eta$, i.e., $\overline{\mathcal{M}_X^{J_{\varepsilon}}}(e_-; \emptyset)\langle \eta \rangle : =ev^{-1}(\eta)$ and $ev:  \overline{\mathcal{M}_{X, 1}^{J_{\varepsilon}}}(e_-; \emptyset) \to \widehat{X}$ is the evaluation map. Note that $$dim \overline{\mathcal{M}_X^{J_{\varepsilon}}}(e_-; \emptyset)\langle \eta \rangle=-1+CZ(e_-) +2+ dim \eta -4 =1.$$Therefore, we have 
\begin{equation*}
\begin{split}
 0=& \#^{vir} \partial \overline{\mathcal{M}_X^{J_{\varepsilon}}}(e_-; \emptyset)\langle \eta \rangle\\
 =&\#^{vir} \overline{\mathcal{M}_X^{J_{\varepsilon}}}(e_-; \emptyset)\langle p \rangle  - \sum_{\gamma_1,...,\gamma_k}  \#^{vir}\left(\overline{\mathcal{M}_Y^{J_{\varepsilon}}}(e_-; \gamma_1,...,\gamma_k)/\mathbb{R}\langle q \rangle\right)\epsilon_X(\gamma_1), ...,\epsilon_X(\gamma_k).
 \end{split}
\end{equation*}
Since $e_-$ has minimal contact action and by degree reason, we have $k=1$ and $\gamma_1=e_-$. Therefore, the only holomorphic curve in   $\overline{\mathcal{M}_Y^{J_{\varepsilon}}}(e_-; e_-)/\mathbb{R}\langle q \rangle$ is the trivial cylinder. However, it cannot pass through the point $q$ due to our choice.  Thus, the moduli space  $\overline{\mathcal{M}_Y^{J_{\varepsilon}}}(e_-; e_-)/\mathbb{R}\langle q \rangle$  is empty, and hence $\epsilon_X\langle p \rangle(e_-)=\#^{vir} \overline{\mathcal{M}_X^{J_{\varepsilon}}}(e_-; \emptyset)\langle p \rangle=0$. This also contradicts the finitness of $c_1(X, \omega_{\varepsilon})$.

Finally, we remark that the above argument can be done without using any virtual techniques, because, one can show that the holomorphic curves  appeared in the above moduli spaces  are simple. 
\end{proof}


By Lemmas  \ref{lem1}, \ref{lem2} and \ref{lem7},   $|\gamma|=2$ if and only if $\gamma =e_+$ or $\gamma=e_-^2$.  The case that $\gamma=e_-^2$ is ruled  out by the  same argument as in Lemma \ref{lem7}. 

\begin{lemma} \label{lem3}
We have  $\epsilon_X\langle p \rangle (e_-^2)=\#^{vir} \overline{\mathcal{M}_X^{J_{\varepsilon}}}(e_-^2; \emptyset)\langle p \rangle =0$.
\end{lemma}
\begin{proof}
Reintroduce the point $q\in Y$ and the path $\eta$ in Lemma \ref{lem7}. 
Then the virtual dimension of $\overline{\mathcal{M}_X^{J_{\varepsilon}}}(e_-^2; \emptyset)\langle \eta \rangle $ is $1$. Therefore, we have 
\begin{equation*}
\begin{split}
 0=& \#^{vir} \partial \overline{\mathcal{M}_X^{J_{\varepsilon}}}(e_-^2; \emptyset)\langle \eta \rangle\\
 =&\#^{vir} \overline{\mathcal{M}_X^{J_{\varepsilon}}}(e_-^2; \emptyset)\langle p \rangle  -  \sum_{\gamma_1,...,\gamma_k} \#^{vir}\left(\overline{\mathcal{M}_Y^{J_{\varepsilon}}}(e_-^2; \gamma_1,...,\gamma_k)/\mathbb{R}\langle q \rangle\right)\epsilon_X(\gamma_1), ...,\epsilon_X(\gamma_k).
 \end{split}
\end{equation*}
Note that $\mathcal{A}_{\lambda_{\varepsilon}}(e_-)<\mathcal{A}_{\lambda_{\varepsilon}}(h_i)<\mathcal{A}_{\lambda_{\varepsilon}}(e_+)$. By energy and degree reasons, $ 1\le  k\le 2$,  and   $\gamma_i$ is either $e_-$, $e_-^2$, $h_i$ or $e_+$.  Also, if there exists $\gamma_i$ such that   $\gamma_i=h_j$ or $\gamma_i=e_+$, then $k=1$. Therefore, the existence of $h_j$ or $e_+$ is  ruled out by the degree reasons. The only possibility is that $k=1$ and $\gamma_1=e_-^2$,  or $k=2$ and $\gamma_1=\gamma_2=e_-$. 
However,  the  energy of the holomorphic curves in  $\overline{\mathcal{M}_Y^{J_{\varepsilon}}}(e_-^2; e_-^2)/\mathbb{R}\langle q \rangle$ or $ \overline{\mathcal{M}_Y^{J_{\varepsilon}}}(e_-^2; e_-, e_-)/\mathbb{R}\langle q \rangle
$
are zero.  Hence, the image of curves are contained in $\mathbb{R} \times e_-$. By our choice of $q$,  $\overline{\mathcal{M}_Y^{J_{\varepsilon}}}(e_-^2; \gamma_1,..,\gamma_k)/\mathbb{R}\langle q \rangle=\emptyset.$  Therefore, we have $$ 0=\#^{vir} \partial \overline{\mathcal{M}_X^{J_{\varepsilon}}}(e_-^2; \emptyset)<\eta> =\#^{vir} \overline{\mathcal{M}_X^{J_{\varepsilon}}}(e_-^2; \emptyset)\langle p \rangle.$$

\end{proof}

Lemma \ref{lem3} and $c_1(X, \omega_{\varepsilon})<L  $ imply  that 
\begin{equation*}
\epsilon_X\langle p \rangle (e_+) = \#^{vir}\overline{\mathcal{M}_X^{J_{\varepsilon}}}(e_+; \emptyset)\langle p \rangle \ne 0.
\end{equation*}

\begin{lemma} \label{lem4}
For a generic $J_{\varepsilon},$ the moduli space $\mathcal{M}_X^{J_{\varepsilon}}(e_+; \emptyset)\langle p \rangle$ is compact and Fredholm regurlar.  In particular, we get 
 \begin{equation*}
 \# {\mathcal{M}_X^{J_{\varepsilon}}}(e_+; \emptyset)\langle p \rangle= \#^{vir}\overline{\mathcal{M}_X^{J_{\varepsilon}}}(e_+; \emptyset)\langle p \rangle \ne 0.
 \end{equation*}
 
\end{lemma}
\begin{proof}
Let $\mathbf{u}=\{u_0, ..., u_N\} \in \overline{\mathcal{M}_X^{J_{\varepsilon}}}(e_+; \emptyset)\langle p \rangle$ be a holomorphic building. No bubbles can   appear due to the aspherical assumption \ref{T2}. 
 By energy and degree reasons, the negative end of the top level $u_N$ is asymptotic to $h_i$ or $e_-$.  By the same reasons, the negative ends  of $u_j$ are simple orbits ($e_-$ or $h_i$) for $1\le  j \le N$. In particular, $u_j $ are  simple  irreducible curves for $0\le j\le N$. 

  Since $u_0$ passes  through the marked point $p$, we have $ind u_0 \ge 2$. By the additivity property of the Fredholm index, $ind u_0=2$ and $ind u_i=0$ for $i \ge 1$. Therefore, $u_i$ are trivial cylinders which are ruled out by the stability condition.   This shows that $\mathcal{M}_X^{J_{\varepsilon}}(e_+; \emptyset)\langle p \rangle =  \overline{\mathcal{M}_X^{J_{\varepsilon}}}(e_+; \emptyset)\langle p \rangle$, and hence  it  is compact.  
Since $J_{\varepsilon}$ is generic,   $\mathcal{M}_X^{J_{\varepsilon}}(e_+; \emptyset)\langle p \rangle$  is a manifold of expected dimension. 
(In fact, we know that $u_0$ is embedded in the next lemma. By Wendl's automatic   transversality  theorem \cite{Wen2} and $2= ind u_0>-2 + \#\Gamma_0(u_0) +2Z(du_0) =-2$,  $\mathcal{M}_X^{J_{\varepsilon}}(e_+; \emptyset)\langle p \rangle$  is still  a manifold  without requiring that $J_{\varepsilon}$ is    generic.)
\end{proof}

\paragraph{Notation} Note that  $\mathcal{M}_X^{J_{\varepsilon}}(e_+; \emptyset)\langle p \rangle $ may be disconnect. By Lemma \ref{lem4},  there is a nonempty connected component. Such   a nonempty connected component is denoted by  $\mathcal{M}_{X}^{J_{\varepsilon}}(e_+; \emptyset)_{\star}\langle p \rangle.$ 

\begin{lemma}  \label{lem5}
The moduli space $\mathcal{M}_{X}^{J_{\varepsilon}}(e_+; \emptyset)_{\star}\langle p \rangle$ consists of a single embedded holomorphic plane. 
\end{lemma}
\begin{proof}
 First, we show that every  holomorphic plane  $u \in \mathcal{M}_{X}^{J_{\varepsilon}}(e_+; \emptyset)_{\star}\langle p \rangle$ is embedded. 
Since $ind u = 2c_{\tau}(u^* TX) =2$, we have  $c_{\tau}(u^* TX) =1$.  
By the relative adjunction formula \cite{H2}, we have 
\begin{equation}  \label{eq2}
\begin{split}
1=c_{\tau}(u^*TX)= &\chi(u)+Q_{\tau}(u) +w_{\tau}(u) -2\delta(u) \\
=&1+Q_{\tau}(u) -2\delta(u),
\end{split}
\end{equation}
where $w_{\tau}(u)$ is the  asymptotic writhe of $u$ (see \cite{H2} for its definition).  Here we use a fact  that $w_{\tau}(u)=0$ if the ends of  $u$ are asymptotic to distinct simple Reeb orbits. 

By Lemma 3.7 of \cite{NW},  $[e^{|e|}_+]=0 \in H_1(Y, \mathbb{Z})$. Thus, it bounds  a 2-chain $Z \in H_2(Y, e_+^{|e|}; \emptyset)$.   By Lemma 3.13 of \cite{NW}, we have $Q_{\tau}(Z) =|e|$.
  We regard $Z$ as a 2-chain  in $H_2(X, e_+^{|e|}; \emptyset)$.  Because   $H_2(X, e_+^{|e|}; \emptyset)$ is an affine space over $H_2(X, \mathbb{Z}), $ there is a  class $A\in H_2(X, \mathbb{Z})$ such that $|e|[u] = Z+A \in H_2(X, e_+^{|e|}; \emptyset)$.  Since $i_*:  H_2(Y, \mathbb{Z}) \to H_2(X, \mathbb{Z})$ is the zero map, we can represented $A$ by a 2-chain away from  $Y$. In particular,  $Z \cdot A =0$. 
By the definition of $Q_{\tau}$, we have 
\begin{equation} \label{eq3}
\begin{split}
|e|^2Q_{\tau}(u) &= Q_{\tau}(|e|u) = Q_{\tau}(Z+A) \\
&= Q_{\tau}(Z)+A\cdot A +2Z \cdot A\\
&= |e| +A\cdot A \le |e|  \ \ (\mbox{by assumption \ref{T4}}). 
 \end{split}
\end{equation}
Equations (\ref{eq2}), (\ref{eq3}) imply that  $2\delta(u) \le 1/|e| \le 1$. Therefore,  $\delta(u) =0$  and    $u$ is embedded.   Moreover, from the above estimates we see that $Q_{\tau}(u)=0$, and the class $A$ satisfies $A\cdot A =e$. 

By  the adjunction formula \cite{S},  we have 
\begin{equation*}
u\bullet u=  2(\delta_{\infty}(u) + \delta(u))+ \frac{1}{2}( ind u -2   +  \# \Gamma_0(u))  +  cov_{\infty}(e_+)=0. 
\end{equation*}
Here $\delta_{\infty}(u) = cov_{\infty}(e_+) =0$ follows from the facts \ref{F3} and \ref{F4}. 
Suppose that we have another holomorphic plane $ v \in \mathcal{M}_{X}^{J_{\varepsilon}}(e_+, \emptyset)_{\star}\langle p \rangle$ other than $u$. Since  $u$ and $v$  in the same component of the moduli space, they are homotopic to each other. Since the intersection number $\bullet$ is homotopic invariant,  $v\bullet  u =u\bullet  u =0$. By \ref{F1},  $v$ and $u$ are disjoint.  However, $p \in v\cap u$. We get a contradiction.  
\end{proof}

\begin{remark}
By degree and energy reason, the moduli space $\overline{\mathcal{M}_Y^{J_{\varepsilon}}}(e_+; \gamma_1,...,\gamma_k)/\mathbb{R}$ is nonempty only when $k=1$ and $\gamma_1$ is $h_i$ or $e_-$. If we require the virtual dimension of  $ \overline{\mathcal{M}_Y^{J_{\varepsilon}}}(e_+; \gamma_1)/\mathbb{R}$  is zero, then $\gamma_1= h_i$.  By Proposition 4.7 of  \cite{NW}, the holomorphic curves in  $ {\mathcal{M}_Y^{J_{\varepsilon}}}(e_+; h_i)$ are 1-1 corresponding to Morse flow lines of $H$.  Since $H$ is perfect, we have $\partial e_+=0$. 

Since $e_+$ is the only Reeb orbit such that $\mathcal{A}_{\lambda_{\varepsilon}}<L$ and $\varepsilon_X\langle p \rangle (e_+) \ne 0$, we have
 $$c_1(X, \omega_{\varepsilon})=\mathcal{A}_{\omega_{\varepsilon}}(e_+)=\int u^* 
 \check{\omega}_{\varepsilon} = 1+ \varepsilon H(p_+) + \omega_{\varepsilon}(A)/|e|,$$
 where $u \in \mathcal{M}_{X}^{J_{\varepsilon}}(e_+; \emptyset)_{\star}\langle p \rangle$ and $A$ is the class in Lemma \ref{lem5}. 
\end{remark}

 
\subsection{Morse-Bott situation}
In this subsection, we consider the holomorphic curves in $(\widehat{X}, \hat{\omega})$. Now $(Y, \lambda)$ is Morse-Bott. The fibers and their iteration are the only Reeb orbits. 

Fix a generic  cobordism admissible almost complex structure $J$ such that $J_{\varepsilon}$ converges to $J$ in $C^{\infty}$-topology as $\varepsilon \to 0$. Let $\mathcal{M}_{X, m}^J(\Sigma)$ denote  the moduli space of holomorphic planes with  $m$ marked points   and the puncture  is asymptotic to a  fiber  $\gamma_z$ over $\Sigma$.  The punctures   of  the curves  are  unconstrained  in the sense of \cite{Wen2}. By Corollary 5.4 of \cite{BO}, the   virtual  dimension of $\mathcal{M}_{X, m}^J(\Sigma)$ is 
\begin{equation*}
\begin{split}
dim \mathcal{M}_{X, m}^J(\Sigma)= &ind u+2m\\
=& -1 + 2c_{\tau}(u^*TX) + \mu_{RS}^{\tau}(\gamma_z) + 1 +2m,\\
\end{split}
\end{equation*}
where $\mu_{RS}^{\tau} $ is the Robbin-Salamon index.  According to Lemma 3.9 of \cite{NW}, $\mu_{RS}^{\tau}(\gamma_z^k) =0$.  By the computations in Lemma \ref{lem5} and Remark \ref{remark1}, we have $c_{\tau}(u^*TX) =1$. Hence, 
$$dim \mathcal{M}_{X, m}^J(\Sigma)=  2c_{\tau}(u^*TX)+2m=2m+2.$$

\begin{remark} \label{remark1}
 Let $u_{\varepsilon} \in \mathcal{M}_X^{J_{\varepsilon}}(e_+, \emptyset)_{\star}\langle p \rangle$  denote the holomorphic plane in Lemma \ref{lem5}.  We know that $c_{\tau}(u_{\varepsilon} ^*TX) =1$ and $Q_{\tau}(u_{\varepsilon} )=0$.  Let  $u_{\varepsilon} \#u_z$,   where $u_z :=\pi^{-1}(\eta)$ and $\eta : [0,1] \to \Sigma$ is a path from $p_+$ to $z$.   For any $u \in \mathcal{M}_{X}^J(\Sigma)$,   by assumption \ref{T2}, we have  
 $c_{\tau}(u^*TX)=c_{\tau}((u_{\varepsilon}\#u_z) ^*TX)= c_{\tau}(u_{\varepsilon}^*TX)=1$. 
If $u$ and $u_{\varepsilon}\#u_z$ represent the same relative homology class, then we also have 
$Q_{\tau}(u)=Q_{\tau}(u_{\varepsilon}) +Q_{\tau}(u_z) =Q_{\tau}(u_{\varepsilon})=0$.   
\end{remark}

\begin{definition} [Definition 4.12, 4.13, 4.14 of \cite{BO}]
A   Morse-Bott building in $\widehat{X}$ is a  chain of rational holomorphic curves $\mathbf{u} =\{u_0,..., u_N\}$ together with a  sequence  of    numbers  $\{T_i \in (0, \infty]\}_{i=0}^{N-1}$, where  $u_0$ is the main level in $\widehat{X}$. $u_i$ satisfy the following conditions.  
\begin{itemize}
\item
For any two adjacent curves $u_i$ and $u_{i+1}$,  the negative ends of $u_{i+1}$ are paired with positive ends of $u_i$. 
Let $(\gamma_{i, j}^+, \gamma^{-}_{i+1, j'})$ be a pair of Reeb orbits.  Then the gradient flow $\varphi_H^{T_i}$ sends $\gamma_{i, j}^+$ to $\gamma^{-}_{i+1, j'}$, where $\varphi_H^t$ is the flow $\partial_t \varphi_H^{t} = \nabla H \circ \varphi_H^{t}. $

\item
Each irreducible component of $u_i$ is either has positive energy or its domain is stable, or $u_i$ contains at least one nonconstant  Morse flow line.  
\end{itemize}
The positive  asymptotics  of  $\mathbf{u}$  are the Reeb  orbits $\varphi_H^{\infty}(\gamma_{N, j}^+)$. 
\end{definition}


\begin{lemma} \label{lem6}
There is a connected component  $\mathcal{M}_{X, 1}^J(\Sigma)_{\star}$ of  $\mathcal{M}_{X, 1}^J(\Sigma)$ such that  evaluation map  $ev: \mathcal{M}_{X, 1}^J(\Sigma)_{\star} \to \widehat{X}$  is a diffeomorphism. 
\end{lemma}  
\begin{proof}
Fix a connected component of  $\mathcal{M}_{X, 1}^J(\Sigma)_{\star}$ such that the holomorphic planes are homotopic to $u_{\varepsilon} \#u_z$ in Remark \ref{remark1}. It suffices to show that $ev$ is a bijective local diffeomorphism.    

Let $u,v \in \mathcal{M}_{X, 1}^J(\Sigma)_{\star}$ be two distinct holomorphic planes. They are simple because the asymptotic ends are simple Reeb orbits. Let $\gamma_z$ be a fiber. Under the   trivialization $\tau$, the asymptotic operator  of $\gamma_z$ is $\mathbf{A}_{\gamma_z} = -J_0 \partial_t$.  Therefore, $\mu_{\tau}(\gamma_z \mp \epsilon) =\pm 1$, where $\epsilon>0$.  $v$ is homotopic to $u$ implies that they have the same relative homology class. By definition and Remark \ref{remark1}, we have 
\begin{equation*}
\begin{split}
&u \bullet u =Q_{\tau}(u) -\Omega_{+}^{\tau}(\gamma_{z}-\epsilon, \gamma_{z}-\epsilon) = \alpha^{\tau}_-(\gamma_{z}-\epsilon)=0,\\
&u \bullet v=Q_{\tau}(u, v) -\Omega_{+}^{\tau}(\gamma_{z}-\epsilon, \gamma_{z'}-\epsilon) =- \min\{ -\alpha^{\tau}_-(\gamma_{z}-\epsilon),  -\alpha^{\tau}_-(\gamma_{z'}-\epsilon)\}=0. 
\end{split}
\end{equation*} 
$u \bullet  v=0$ implies that $u$ and $v$ are disjoint. By the adjunction formula  \cite{Wen2}, we have 
\begin{equation*}
\begin{split}
0=u \bullet  
u = &2\delta_{total}(u)+ \frac{1}{2}( ind u -2 +2g(u) + \# \Gamma_0(u))  + cov_{\infty}(\gamma) + cov_{MB}(\gamma)\\
=& 2\delta_{total}(u) + \frac{1}{2}( ind u -2 +2g(u) + \# \Gamma_0(u)) =2 \delta_{total}(u). 
\end{split}
\end{equation*} 
Here $cov_{\infty}(\gamma) = cov_{MB}(\gamma)=0$ follows from \ref{F4}. 
Therefore, the holomorphic plane is embedded. 
Since  the  holomorphic planes  in  $\mathcal{M}_{X, 1}^J(\Sigma)_{\star}$ are embedded and pairwise  disjoint, $ev$ must be injective. 
 
To see $ev$ is surjective, we show that its image is open and closed.  Let $ u \in  \mathcal{M}_{X, 1}^J(\Sigma)_{\star}$.  By Theorem 4.5.42 of  \cite{Wen1}, there is  small   neighborhood of $u$ is foliated by the index 2 holomorphic planes. More precisely, we have an embedding 
\begin{equation} \label{eq1}
\begin{split}
G: &B_{\delta}^2(0) \times \mathbb{C} \to \widehat{X}\\
&(w,z) \to \Psi_u(\eta_w(z)),
\end{split}
\end{equation}
where $\eta_w$ is a family of sections of the normal bundle  $N_u$ parametrized by a small disk.  Here $\Psi_u$ is a  trivialization of the normal bundle such that   $\Psi_u =exp_u$ away from the  end. 
 Therefore, the image of $ev$ is open. Moreover, $G$ is embedding implies that $ev$ is a local diffeomorphism.

 We claim that the image of $ev$ is closed.  
Let $\{(u_n, z_n)\}_{n=1}^{\infty} \subset \mathcal{M}_{X, 1}^J(\Sigma)_{\star}$  be a sequence of holomorphic planes  such that $ev((u_n, z_n)) = u_n(z_n) =p_n$. Suppose that  $\lim_{n \to \infty } p_n = p \in \widehat{X}$.   We want to show that $p$  also lies  inside the image of $ev$.

Since $\int u_n^* \check{\omega} =A_{\omega}(\gamma_{z_n})$ is  independent of $n$, by Gromov compactness \cite{BEHWZ}, the sequence $\{u_n\}_{n=1}^{\infty} $ converges to     a   holomorphic building  $\mathbf{u}=\{u_0, ..., u_N\}$.   Because $\gamma_{z_n}$ have  minimal   contact action, the ends of  $\{u_n\}_{n=1}^{\infty} $ cannot move to    multiple covered Reeb orbits. 
The top level $u_N$ is asymptotic to $\gamma_z$ for some $z \in \Sigma$.  By the degree reason, $u_N$ has a nonempty negative end $\gamma'$.  The energy of $u_N$ is 
$$\int u_N^* d\lambda = \mathcal{A}_{\lambda}(\gamma_z) -  \mathcal{A}_{\lambda}(\gamma') \ge 0. $$
Because $\gamma_z$ has the  minimal contact action,    we must have  $\int u_N^* d\lambda  =0$.  Hence, $u_N$ is  the trivial cylinder.  

Note that the marked point  cannot lie inside the cylindrical levels  because $ev((\mathbf{u}, z_{\infty})) =p \in \widehat{X}$.  Therefore, $u_N$ is a trivial cylinder without mark point which is ruled out the by stability condition. 
 By induction,   $\mathbf{u}$ only consists of the cobordism level. No bubbles can appear due to the symplectically aspherical assumption \ref{T2}. Thus, $(\mathbf{u}, z_{\infty})$ is a holomorphic plane $(u_0, z_{\infty}) \in \mathcal{M}_{X, 1}^J(\Sigma)_{\star} $. Moreover, $p=ev((\mathbf{u}, z_{\infty}))=u_0(z_{\infty})$.  

 In sum, the image of $ev$ is either the empty set or $\widehat{X}$.  By Lemma \ref{lem5}, there is a holomorphic plane  $u_{\varepsilon} \in \mathcal{M}_X^{J_{\varepsilon}}(e_+; \emptyset)_{\star}\langle p \rangle$   for each $0<\varepsilon < \varepsilon_L$ (a prior, the component $\star$ may depend on $\varepsilon$ ).  The Morse-Bott compactness (Section 4.2.2) in \cite{BO} can be adapted to our cobordism setting. As  $\varepsilon \to 0$,  $u_{\varepsilon}$ converges to a Morse-Bott  holomorphic  building $\mathbf{u}$   such that $ev(\mathbf{u}) =p$ and $\mathbf{u}$ is asymptotic to the  Reeb orbit over $p_+$.  As $\gamma_{p_+}$ has minimal symplectic/contact action, by the same argument as before, each  positive level only consists of  a trivial cylinder  over $\gamma_{z}$ together with Morse flow lines.  Also, there is no bubbles because of \ref{T2}. Therefore, the main level $u_0$ of $\mathbf{u}$ is a holomorphic plane in  $\mathcal{M}_{X, 1}^J(\Sigma)$ such that $ev(u_0)=p$.  
In sum, $ev$  maps  $\mathcal{M}_{X, 1}^J(\Sigma)_{\star}$  to the whole $\widehat{X}$.
 \end{proof}

Let $\mathcal{M}_{X}^J(\Sigma)_*$  be the  connected component  by forgetting the marked point of  $\mathcal{M}_{X, 1}^J(\Sigma)_{\star}$.  In particular, it is nonempty. 
\begin{lemma}
The moduli space $\mathcal{M}_{X}^J(\Sigma)_{\star}$ is a   compact  orientable   manifold of dimension two.   In other words, $\mathcal{M}_{X}^J(\Sigma)_{\star}$ is a closed connected orientable  surface. 
\end{lemma}
\begin{proof}
By the same energy and degree  reasons  as the proof of Lemma \ref{lem6}, the   positive levels of the broken holomorphic curve $\mathbf{u} \in \overline{\mathcal{M}_{X}^J}(\Sigma)_{\star}$ must be trivial cylinders, which  are ruled out by the  stability condition. Again by \ref{T2}, no bubbles exist.  Therefore,   $\mathcal{M}_{X}^J(\Sigma)_{\star}$ is compact.  

For $u  \in \mathcal{M}_{X}^J(\Sigma)_{\star}$, we know that $u$ is embedded from the proof of Lemma \ref{lem6}. Moreover, we have $2 =ind u>  -2 + \#\Gamma_0(u) +2Z(du) =-2$. The transversality  follows from Wendl's automatic transversality theorem \cite{Wen2}.  By Lemma 7.2 of \cite{BO}, $\mathcal{M}_{X}^J(\Sigma)$ is orientible. 
\end{proof}

Define  a surjective map  $\mathfrak{F}:  \mathcal{M}_{X, 1}^J(\Sigma)_{\star} \to  \mathcal{M}_{X}^J(\Sigma)_{\star}$ by forgetting the marked  point. The map $\mathfrak{F}$ is smooth by definition.   Define a map by 
\begin{equation*}
\begin{split}
\Pi:   \widehat{X} \to  \mathcal{M}_{X}^J(\Sigma)_{\star}, \ \ \ x \to \mathfrak{F} \circ ev^{-1}(x). 
\end{split}
\end{equation*}
We complete the proof of Theorem \ref{thm0} by the following lemma. It shows that $\widehat{X}$ is isomorphic to the associated vector bundle $E$. In particular, $X$ is diffeomorphic to the associated disk bundle $\mathbb{D}E$. 
\begin{lemma}
   $\Pi: \widehat{X} \to \mathcal{M}_{X}^J(\Sigma)_{\star}$ is a fiber bundle with fiber $\mathbb{R}^2$. Moreover,  $\widehat{X}$ is isomorphic (as a fiber  bundle with structure group $Diff(\mathbb{R}^2)$) to a  line  bundle $\pi_E : E \to  \Sigma$ with Euler number $e$. 
\end{lemma}
\begin{proof}
The local trivialization of  of $\Pi: \widehat{X} \to  \mathcal{M}_{X}^J(\Sigma)_{\star}$  is given by the local foliation (\ref{eq1}).  By definition,  we have the following diagram:  
$$\begin{CD}
				 B_{\delta}^2(0) \times \mathbb{C} @>G >>  \widehat{X} \\
				@VV pr  V @VV  \Pi=\mathfrak{F} \circ ev^{-1} V\\
				 B_{\delta}^2(0)    @> g >>   \mathcal{M}_{X}^J(\Sigma)_{\star}
			\end{CD}$$
Here $g$ maps $w \in  B_{\delta}^2(0) $ to  $\Psi_u(\eta_w) \in \mathcal{M}_{X}^J(\Sigma)_{\star}$.  Therefore, $\Pi: \widehat{X} \to \mathcal{M}_{X}^J(\Sigma)_{\star}$ is a fiber bundle with fiber $\mathbb{R}^2$.
By the Corollary.I of \cite{ES},  $\widehat{X}$ is isomorphic  to a  line bundle $\pi_E : E \to  \mathcal{M}_{X}^J(\Sigma)_{\star}$.     In particular, $\widehat{X}$  is diffeomorphic  to the total space $E$.

Since $\widehat{X}$  is diffeomorphic  to the total space $E$, we have 
\begin{equation*}
H_*(X, \mathbb{Z}) \cong  H_* (E, \mathbb{Z}) \cong H_*(\mathcal{M}_{X}^J(\Sigma)_{\star}, \mathbb{Z})  \mbox{ and } H^*(X, \mathbb{Z}) \cong  H^*(E, \mathbb{Z}) \cong H^*(\mathcal{M}_{X}^J(\Sigma)_{\star}, \mathbb{Z}).  
\end{equation*}
The above isomorphisms are  also true for $\mathbb{R}$ coefficient.   By the exact sequence for relative homology and the above isomorphism, we have 
$$  \mathbb{R}^{2g} \cong H_1(Y, \mathbb{R}) \cong H_1(X, \mathbb{R}) \cong  H_1(\mathcal{M}_{X}^J(\Sigma)_{\star}, \mathbb{R}). $$  
The above isomorphism $H_1(Y, \mathbb{R}) \cong \mathbb{R}^{2g} $ follows from  Lemma 3.7 of \cite{NW}. Therefore, $\mathcal{M}_{X}^J(\Sigma)_{\star}$ is diffeomorphic to  $\Sigma$. 

Reintroduce the class $A$ in Lemma \ref{lem5}   such that $|e| [u]=Z+A$, where $u \in \mathcal{M}^{J_{\varepsilon}}_X(e_+; \emptyset)_{\star}\langle p \rangle$. The class $A$ satisfies 
$A \cdot A =e$ and $Z\cdot A=0$. We write  $A=  k [\mathcal{M}_{X}^J(\Sigma)_{\star}]$ for some $k \in \mathbb{Z}$.  
Let   $S_0\cong \mathcal{M}_{X}^J(\Sigma)_{\star}$   denote the zero section of $E$. 
Note that $u \cdot S_0 =(Z+A) \cdot S_0/|e| =-1/k$ is an integer. Therefore, $k=\pm 1$. The Euler number of $E$ is $$<c_1(E),  [S_0]> =  [S_0] \cdot [S_0] =A \cdot A =e. $$   

In fact, we must have $k=-1$.  By the computations in Lemma \ref{lem5} and Lemma 3.12 of \cite{NW},  we have  $ <c_1(TX), A> =|e|c_{\tau}(u^*TX)-c_{\tau}(Z)= |e|-\chi(\Sigma).$
  The line bundle $E$ admits a symplectic form $\Omega$  such that  $S_0$ is an embedded symplectic surface with $\int_{S_0} \Omega>0$. By the adjunction formula, we have 
\begin{equation*}
\chi(\Sigma)= \chi(S_0) = \frac{1}{k}<c_1(TX), A> -\frac{1}{k^2}A\cdot A = \frac{1}{k}(-\chi(\Sigma) +|e|) +\frac{1}{k^2}|e|.
\end{equation*}
The above equation implies that $k=-1$.  

 
\end{proof}

Shenzhen University

\href{mailto: ghchen@szu.edu.cn} {\tt ghchen@szu.edu.cn} \\

\end{document}